\documentclass[11pt]{article}

\usepackage{amsmath,amsthm,amssymb}
\usepackage{epsfig}
\usepackage{graphicx}
\graphicspath{{pictures/}} 

\usepackage{geometry}
\newtheorem{thm}{Theorem}

\numberwithin{equation}{section}

\usepackage{booktabs} 
\usepackage{cite} 

\title{An optimal three-point eighth-order iterative method without memory for
solving nonlinear equations with its dynamics}

\author{Gunar Matthies $^a$\thanks{gunar.matthies@tu-dresden.de}
\and Mehdi Salimi $^{b,c}$\thanks{mehdi.salimi@tu-dresden.de} \and
Somayeh Sharifi $^{d}$\thanks{somayeh.sharifi@medalics.org} \and
Juan Luis Varona $^{e}$\thanks{jvarona@unirioja.es}\\[9pt]
\small
$^{a}$Institut f{\"u}r Numerische Mathematik, Technische Universit{\"a}t Dresden, Germany\\[-3pt]
\small
$^{b}$Center for Dynamics, Department of Mathematics, Technische Universit{\"a}t Dresden, Germany\\[-3pt]
\small
$^{c}$Department of Mathematics, Universiti Putra Malaysia, 43400 UPM Serdang, Selangor, Malaysia\\[-3pt]
\small
$^{d}$MEDAlics, Research \small Center at Universit\`{a} per
Stranieri Dante Alighieri, Reggio Calabria, Italy\\[-3pt]
\small $^{e}$Departamento de Matem{\'a}ticas y Computaci{\'o}n,
Universidad de La Rioja, Logro{\~n}o, Spain }

\date{}

\begin{document}

\maketitle

\begin{abstract}
We present a three-point iterative method without memory for
solving nonlinear equations in one variable. The proposed method
provides convergence order eight with four function evaluations
per iteration. Hence, it possesses a very high computational
efficiency and supports Kung and Traub's conjecture. The
construction, the convergence analysis, and the numerical
implementation of the method will be presented. Using several test
problems, the proposed method will be compared with existing
methods of convergence order eight concerning accuracy and basin
of attraction. Furthermore, some measures are used to judge
methods with respect to their performance in finding the basin of
attraction.
\medskip

\medskip

\end{abstract}

\section{Introduction}
Solving nonlinear equations is a basic and extremely valuable tool
in all fields in science and engineering. One can distinguish
between two general approaches for solving nonlinear equations
numerically, namely, one-point and multi-point methods. The basic
optimality theorem shows that an analytic one-point method based on $k$
evaluations is of order at most $k$, see~\cite[\S\,5.4]{Traub}
or~\cite{Kung} for an improved proof. The Newton--Raphson method
$x_{n+1}:=x_n - \frac{f(x_n)}{f'(x_n)}$ is probably the most widely used
algorithms for finding roots. It requires two evaluations per iteration
step, one for $f$ and one for $f'$, and results in second order
convergence which is optimal for this one-point method.

Some computational issues encountered by one-point methods are overcome by
multi-point methods since they allow to achieve greater accuracy with the
same number of function evaluations. Important aspects related to these
methods are convergence order and efficiency. It is favorable to attain
with a fixed number of function evaluations per iteration step a
convergence order which is as high as possible. A central role in this
context plays the unproved conjecture by Kung and Traub~\cite{Kung} which
states that an optimal multi-point method without memory provides a
convergence order of~$2^k$ while using $k+1$ evaluations in each iteration
step. The efficiency index for a method with $k$ evaluations and convergence
order $p$ and $k$ evaluations is given by $E(k,p)=\sqrt[k]{p}$,
see~\cite{Ostrowski}. Hence, the efficiency of a method supporting Kung and
Traub's conjecture is $\sqrt[k+1]{2^k}$. In particular, an optimal method
with convergence order eight has an efficiency index
$\sqrt[4]{8}\simeq 1.68179$.

A large number of multi-point methods for finding simple roots of a
nonlinear equation $f(x)=0$ with a scalar function
$f: D\subset\mathbb{R} \to \mathbb{R}$ which is defined on
an open interval~$D$ (or $f: D\subset\mathbb{C} \to \mathbb{C}$
defined on a region $D$ in the complex plane~$\mathbb{C}$) have been developed
and analyzed for improving the convergence order of classical methods like
the Newton--Raphson iteration.

Some well known two-point methods without memory are described
e.g.\ in Jarratt~\cite{Jarratt}, King~\cite{King}, and
Ostrowski~\cite{Ostrowski}. Using inverse interpolation, Kung and
Traub \cite{Kung} constructed two general optimal classes without
memory. Since then, there have been many attempts to construct
optimal multi-point methods, utilizing e.g.\ weight functions, see
in particular~\cite{Babajee, Bi, Chun1, Lotfi2, Petkovic,
Sharifi1, Sharifi2, Sharifi3, Sharma1, Thukral, Wang}.

We will construct a three-point method of convergence order eight
which is free from second order derivatives, uses $4$ evaluations, and
provides the efficiency index $\sqrt[4]{8} \simeq 1.68179$.

A wide used criterion to judge and rank different methods for
solving nonlinear equations is the basin of attraction. We will
use two measures to assess the performance in finding the basin of
attraction~\cite{Varona}.

The paper is organized as follows. Section~\ref{sec:description}
introduces the new methods based on a Newton step and Newton's
interpolation. Moreover, details of the new method and the proof
of its optimal convergence order eight are given. The numerical
performance of the proposed method compared to other methods are
illustrated in Section~\ref{sec:examples}. We approximate and
visualize the basins of attraction in Section~\ref{sec:dynamic}
for the proposed method and several existing methods, both
graphically and by mean of introduced numerical performance
measures~\cite{Varona}. Finally, we conclude in
Section~\ref{sec:conclusion}.

\section{Description of the method and convergence analysis}
\label{sec:description}

We construct in this section a new optimal three-point method for
solving nonlinear equations by using a Newton-step and Newton's
interpolation polynomial of degree three which was also applied
in~\cite{Sharifi2}.

\paragraph{Method 1:}
The new method is given by
\begin{equation}
\label{a1}
\left\{
\begin{aligned}
y_{n} & := x_{n} - u_n(x_n), \\[1.2ex]
z_{n} & := x_{n} -
u_n(x_n) \left(1+\dfrac{f(y_n)}{f(x_n)}
+ \left(1+\dfrac{1}{1+u_n(x_n)}\right) \left(\dfrac{f(y_n)}{f(x_n)}\right)^2\right),
\\[1.2ex]
x_{n+1} & := z_{n} -
\dfrac{f(z_n)}{f[z_n,y_n]+(z_n-y_n)f[z_n,y_n,x_n]+(z_n-y_n)(z_n-x_n)f[z_n,y_n,x_n,x_n]},
\end{aligned}
\right.
\end{equation}
where $u_n(x_n)=\frac{f(x_{n})}{f'(x_{n})}$.
The standard notation for divided differences
in Newton's interpolation
\[
  g[t_{\nu},t_{\nu+1}, \dots, t_{\nu+j}]
  = \frac{g[t_{\nu+1}, \dots, t_{\nu+j}] - g[t_\nu, \dots,
  t_{\nu+j-1}]}{t_{\nu+j}-t_{\nu}},
\]
with $g[t_\nu] = g(t_\nu)$ and $g[t_{\nu},t_{\nu}] = g'(t_{\nu})$ are used.

The iteration method~\eqref{a1} and all forthcoming methods are applied for
$n = 0, 1, \dots$ where $x_0$ denotes an initial approximation of the simple
root~$x^{\ast}$ of the function $f$. The method~\eqref{a1} uses four
evaluations per iteration step, three for $f$ and one for $f'$. Note
that~\eqref{a1} works for real and complex functions.

The convergence order of method~\eqref{a1} is given in the following theorem.

\begin{thm}
\label{thm:main}
Let $f:D\subset \mathbb{R} \to \mathbb{R}$ be an
eight times continuously differentiable function with a simple
zero $x^{\ast}\in D$. If the initial point $x_{0}$ is sufficiently
close to $x^{\ast}$ then the method defined by \eqref{a1} converges
to $x^{\ast}$ with order eight.
\end{thm}

\begin{proof}
Let $e_{n} := x_{n}-x^{\ast}$, $e_{n,y}:=y_{n}-x^{*}$,
$e_{n,z} := z_{n}-x^{\ast}$ and
$c_{n} := \frac{f^{(n)}(x^{\ast})}{n!\,f'(x^{*})}$ for $n\in \mathbb{N}$.
Using the fact that $f(x^{\ast})=0$, the Taylor expansion of $f$ at
$x^{\ast}$ yields
\begin{equation}
\label{a10}
f(x_{n}) = f'(x^{\ast}) \left(e_{n}
+ c_{2}e_{n}^{2}+c_{3}e_{n}^{3} + \cdots
+ c_{8}e_{n}^{8}\right) + O(e_n^{9})
\end{equation}
and
\begin{equation}
\label{a11}
f'(x_{n}) = f'(x^{\ast}) \left(1 +
2c_{2}e_{n}+3c_{3}e_{n}^{2}+4c_{4}e_{n}^{3} + \cdots
+ 9c_9e_n^8\right) + O(e_n^{9}).
\end{equation}
Therefore, we have
\begin{equation*}
\begin{split}
\frac{f(x_{n})}{f'(x_{n})}
&= e_{n}-c_{2}e_{n}^{2} + \left(2c_{2}^{2}-2c_{3}\right) e_{n}^{3}
+ \left(-4c_2^3+7c_2c_3-3c_4\right)e_n^4 \\
& \qquad + \left(8c_2^4-20c_2^2c_3+6c_3^2+10c_2c_4-4c_5\right)e_n^5 \\
& \qquad + \left(-16c_2^5+52c_2^3c_3-28c_2^2c_4
+ 17c_3c_4-c_2(33c_3^2-13c_5)\right)e_n^6 + O(e_n^{7}),
\end{split}
\end{equation*}
and
\begin{equation*}
\begin{split}
e_{n,y} = y_n-x^{\ast} & = c_{2}e_{n}^{2} + \left(-2c_2^2+2c_3\right)e_n^3
+ \left(4c_2^3-7c_2c_3+3c_4\right)e_n^4\\
& \qquad + \left(-8c_2^4+20c_2^2c_3-6c_3^2-10c_2c_4+4c_5\right)e_n^5\\
& \qquad + \left(16c_2^5-52c_2^3c_3+28c_2^2c_4-17c_3c_4
+ c_2(33c_3^2-13c_5)\right)e_n^6 + O(e_n^{7}).
\end{split}
\end{equation*}
We have
\begin{equation}
\label{a12a}
f(y_{n}) = f'(x^{\ast}) \left(e_{n,y} + c_{2}e_{n,y}^{2} + c_{3}e_{n,y}^{3}
+ \cdots + c_{8}e_{n,y}^{8}\right) + O(e_{n,y}^{9})
\end{equation}
by a Taylor expansion of $f$ at $x^{\ast}$.
By substituting \eqref{a10}--\eqref{a12a} into \eqref{a1}, we get
\begin{equation*}
\begin{split}
\label{a13a}
e_{n,z} = z_n-x^{\ast}
&= c_2 \left(c_2+5c_2^2-c_3\right)e_n^4 \\
& \qquad + \left(-8c_2^3-36c_2^4-2c_3^2+c_2^2(-1+32c_3) + c_2(4c_3-2c_4)\right)e_n^5 + O(e_n^6).
\end{split}
\end{equation*}
We obtain
\begin{equation}
\label{a12b}
f(z_{n}) = f'(x^{\ast}) \left(e_{n,z} +
c_{2}e_{n,z}^{2}+c_{3}e_{n,z}^{3}+\cdots+c_{8}e_{n,z}^{8}\right) + O(e_{n,z}^{9})
\end{equation}
by using again a Taylor expansion of $f$ at $x^{\ast}$.
Substituting \eqref{a10}--\eqref{a12b} into \eqref{a1}, we get
\begin{equation}
\label{erreq}
e_{n+1} = x_{n+1}-x^{\ast} =
c_2^2 \left(c_2+5c_2^2-c_3\right) \left(c_2^2-5c_2^3-c_2c_3+c_4\right)e_n^8 + O(e_n^9),
\end{equation}
which finishes the proof of the theorem.
\end{proof}

We will compare the new method~\eqref{a1} with some existing
optimal three-point methods of order eight having the same optimal
computational efficiency index equal to $\sqrt[4]{8} \simeq
1.68179$, see~\cite{Ostrowski, Traub}.

The existing methods that we are going to use to compare are the following:

\paragraph{Method 2:}
The method by Chun and Lee \cite{Chun1} is given by
\begin{equation}
\label{o1}
\left\{
\begin{aligned}
y_{n} & := x_{n} - \dfrac{f(x_{n})}{f'(x_{n})}, \\[0.8ex]
z_{n} & := y_{n} - \dfrac{f(y_n)}{f'(x_n)}
\cdot \dfrac{1}{\left(1-\frac{f(y_n)}{f(x_n)}\right)^2}, \\[0.8ex]
x_{n+1} & := z_{n} - \dfrac{f(z_n)}{f'(x_n)}
\cdot \dfrac{1}{\left(1-H(t_n)-J(s_n)-P(u_n)\right)^2}
\end{aligned}
\right.
\end{equation}
with weight functions
\begin{equation*}
H(t_n) = -\beta-\gamma+t_n+\frac{t_n^2}{2}-\frac{t_n^3}{2}, \quad
J(s_n) = \beta+\frac{s_n}{2}, \quad P(u_n) = \gamma+\frac{u_n}{2},
\end{equation*}
where $t_n = \frac{f(y_n)}{f(x_n)}$, $s_n = \frac{f(z_n)}{f(x_n)}$,
$u_n = \frac{f(z_n)}{f(y_n)}$, and $\beta, \gamma \in \mathbb{R}$.
Note that the parameters $\beta$ and $\gamma$ cancel when used
in~\eqref{o1}. Hence, their choice has no contribution to the method.

\paragraph{Method 3:}
The method by B. Neta~\cite{Neta0}, see also~\cite[formula~(9)]{Neta1}, is
given by
\begin{equation}
\label{o2}
\left\{
\begin{aligned}
y_{n} & := x_{n} - \dfrac{f(x_n)}{f'(x_n)}, \\[0.8ex]
z_{n} & := y_{n} - \dfrac{f(x_n)+Af(y_n)}{f(x_n)+(A-2)f(y_n)} \cdot
\dfrac{f(y_n)}{f'(x_n)},
\quad A\in \mathbb{R}, \\[0.8ex]
x_{n+1} & := y_n + \delta_1f^2(x_n)+\delta_2f^3(x_n),
\end{aligned}
\right.
\end{equation}
where
\begin{alignat*}{2}
F_y & = f(y_n)-f(x_n), &\qquad F_z &= f(z_n)-f(x_n),\\
\zeta_y & = \dfrac{1}{F_y}
\left(\dfrac{y_n-x_n}{F_y}-\dfrac{1}{f'(x_n)}\right),
&\qquad
\zeta_z & = \dfrac{1}{F_z}
\left(\dfrac{z_n-x_n}{F_z}-\dfrac{1}{f'(x_n)}\right),\\
\delta_2 & = -\dfrac{\zeta_y-\zeta_z}{F_y-F_z},
&\qquad
\delta_1 & = \zeta_y+\delta_2F_y.
\end{alignat*}
We will use $A = 0$ in the numerical experiments of this paper.

\paragraph{Method 4:}
The Sharma and Sharma method \cite{Sharma1} is given by
\begin{equation}
\label{o3}
\left\{
\begin{aligned}
y_{n} & := x_{n} - \dfrac{f(x_n)}{f'(x_n)}, \\[0.8ex]
z_{n} & := y_{n} - \dfrac{f(y_n)}{f'(x_n)} \cdot \dfrac{f(x_n)}{f(x_n)-2f(y_n)}, \\[0.8ex]
x_{n+1} & := z_{n} - \dfrac{f[x_n,y_n]f(z_n)}{f[x_n,z_n]f[y_n,z_n]}\,W(t_n),
\end{aligned}
\right.
\end{equation}
with the weight function
\begin{equation*}
W(t_n) = 1+\frac{t_n}{1+\alpha t_n},
\quad \alpha\in \mathbb{R},
\end{equation*}
and $t_n = \frac{f(z_n)}{f(x_n)}$.
We will use $\alpha=1$ in the numerical experiments of this paper.

\paragraph{Method 5:}
The method from Babajee, Cordero, Soleymani and Torregrosa \cite{Babajee}
is given by
\begin{equation}
\label{o4}
\left\{
\begin{aligned}
y_{n} & := x_{n} - \dfrac{f(x_n)}{f'(x_n)}
\left( 1 + \left(\dfrac{f(x_n)}{f'(x_n)}\right)^5 \right), \\[0.8ex]
z_{n} & := y_{n} - \dfrac{f(y_n)}{f'(x_n)}
\left( 1 - \dfrac{f(y_n)}{f(x_n)}\right)^{-2}, \\[0.8ex]
x_{n+1} & := z_n - \dfrac{f(z_n)}{f'(x_n)}
\cdot \dfrac{1+\left(\dfrac{f(y_n)}{f(x_n)}\right)^2 + 5\left(\dfrac{f(y_n)}{f(x_n)}\right)^4
+ \dfrac{f(z_n)}{f(y_n)}}
{\left( 1 - \dfrac{f(y_n)}{f(x_n)} - \dfrac{f(z_n)}{f(x_n)} \right)^2} .
\end{aligned}
\right.
\end{equation}

\paragraph{Method 6:}
The method from Thukral and Petkovi\'{c} \cite{Thukral} is given by
\begin{equation}
\label{o5}
\left\{
\begin{aligned}
y_{n} & := x_{n} - \dfrac{f(x_n)}{f'(x_n)}, \\[0.8ex]
z_{n} & := y_{n} - \dfrac{f(y_n)}{f'(x_n)}\cdot
\dfrac{f(x_n)+\beta f(y_n)}{f(x_n)+(\beta-2)f(y_n)}, \qquad \beta\in \mathbb{R}, \\[0.8ex]
x_{n+1} & := z_n - \dfrac{f(z_n)}{f'(x_n)} \cdot
\left(\varphi(t_n)+\psi(s_n)+\omega(u_n)\right),
\end{aligned}
\right.
\end{equation}
where weight functions are
\begin{equation*}
\varphi(t_n) = \left(1+\dfrac{t_n}{1-2t_n}\right)^2,
\qquad \psi(s_n) = \dfrac{s_n}{1-\alpha s_n},
\quad \alpha \in\mathbb{R},\qquad \omega(u_n)=4u_n,
\end{equation*}
and $t_n=\frac{f(y_n)}{f(x_n)}$, $s_n=\frac{f(z_n)}{f(y_n)}$
and $u_n=\frac{f(z_n)}{f(x_n)}$. We will use $\beta=0$ and
$\alpha=1$ in the numerical experiments of this paper.

\section{Numerical examples}
\label{sec:examples}

The new three-point method~\eqref{a1} is tested on several
nonlinear equations. To obtain high accuracy and avoid the loss of
significant digits, we employed multi-precision arithmetic with
20\,000 significant decimal digits in the programming package
Mathematica.

We are going to perform numerical experiments with the four test functions
$f_j$, $j=1,\dots,4$, which appear in Table~\ref{table1}. We are going to
reach the given root $x^{\ast}$ starting with the mentioned $x_0$ for the
four functions and the six methods of convergence order eight.

\begin{table}[htb!]
\begin{center}
\begin{tabular}{lcc}
   \toprule
    test function $f_j$ & root $x^{\ast}$ & initial guess $x_0$
  \\ \midrule
    $f_1(x) = \ln (1+x^2)+e^{x^2-3x}\sin x$ & $0$ & $0.35$
  \\[0.5ex]
  $f_2(x) = 1+e^{2+x-x^2}+x^3-\cos(1+x)$ & $-1$ & $-0.3$
  \\[0.5ex]
   $f_3(x) = (1+x^2)\cos\frac{\pi x}{2}+\frac{\ln(x^2+2x+2)}{1+x^2}$ & $-1$ & $-1.1$
  \\[0.5ex]
   $f_4(x) = x^4+\sin\frac{\pi}{x^2}-5$ & $\sqrt{2}$ & $1.5$
  \\[0.5ex]
  \bottomrule
\end{tabular}
\caption{Test functions $f_1, \dots, f_4$, root
$x^{\ast}$, and initial guess $x_0$.}
\label{table1}
\end{center}
\end{table}

In order to test our proposed method~\eqref{a1} and
compare it with the methods \eqref{o1}--\eqref{o5},
we compute the error, the computational
order of convergence (COC) by the approximate formula~\cite{coc}
\begin{equation}
\label{coc}
\textup{COC} \approx \frac{\ln|(x_{n+1}-x^{\ast})/(x_{n}-x^{*})|}{\ln|(x_{n}-x^{*})/(x_{n-1}-x^{*})|},
\end{equation}
and the approximated computational order of convergence (ACOC) by
the formula~\cite{acoc}
\begin{equation}
\label{acoc}
\textup{ACOC} \approx \frac{\ln|(x_{n+1}-x_{n})/(x_{n}-x_{n-1})|}{\ln|(x_{n}-x_{n-1})/(x_{n-1}-x_{n-2})|}.
\end{equation}

It is worth noting that COC has been used in the recent years.
Nevertheless, ACOC is more practical because it does not require to know the
root~$x^{\ast}$. See~\cite{acoc-et-al} for a comparison among several
convergence orders. Note that these formulas may result for particular
examples in convergence orders which are higher than expected. The reason
is that the error equation~\eqref{erreq} contains problem-dependent
coefficients which may vanish for some nonlinear functions $f$. However,
the formulas~\eqref{coc} and~\eqref{acoc} will provide for a ``random''
example good approximations for the convergence order of the method.


We have used both COC and ACOC to check the accuracy of the considered
methods. Note that both COC and ACOC give already for small values of $n$
good experimental approximations to convergence order.


The comparison of our method~\eqref{a1} with the methods
\eqref{o1}--\eqref{o5} applied to the four nonlinear equations $f_j(x)=0$,
$j=1,\dots,4,$ are presented in in Table~\ref{table2-3}. We
abbreviate~\eqref{a1} by M1 and \eqref{o1}--\eqref{o5} as M2--M6,
respectively. The computational convergence order COC and ACOC are given
$n=3$. Note that they are for all problems and methods in excellent with
the theoretical order of convergence.

\begin{table}[htb!]\small
\begin{center}
\begin{tabular}{cllllll}
 \toprule
  & \multicolumn{1}{c}{M1} & \multicolumn{1}{c}{M2}
  & \multicolumn{1}{c}{M3} & \multicolumn{1}{c}{M4} & \multicolumn{1}{c}{M5}& \multicolumn{1}{c}{M6} \\
\midrule
$f_1$, $x_0=0.35$\\
$|x_{1}-x^{\ast}|$ & $0.610\mathrm{e}{-}6$ & $0.721\mathrm{e}{-}4$
& $0.893\mathrm{e}{-}4$ & $0.753\mathrm{e}{-}4$ & $0.347\mathrm{e}{-}3$& $0.328\mathrm{e}{-}3$ \\
$|x_{2}-x^{\ast}|$ & $0.319\mathrm{e}{-}46$& $0.230\mathrm{e}{-}30$
& $0.126\mathrm{e}{-}30$ & $0.619\mathrm{e}{-}31$ & $0.471\mathrm{e}{-}25$ & $0.256\mathrm{e}{-}25$\\
$|x_{3}-x^{\ast}|$ & $0.179\mathrm{e}{-}368$ &
  $0.252\mathrm{e}{-}242$
& $0.200\mathrm{e}{-}245$ & $0.128\mathrm{e}{-}247$ & $0.546\mathrm{e}{-}200$& $0.345\mathrm{e}{-}202$ \\
COC &  $8.0000$ & $8.0000$ & $8.0000$ & $8.0000$ & $8.0000$ & $8.0000$\\
ACOC & $8.0000$ & $7.9999$ & $7.9999$ & $7.9999$ & $7.9999$ & $7.9999$\\
\midrule
$f_2$, $x_0=-0.3$\\
$|x_{1}-x^{\ast}|$ & $0.248\mathrm{e}{-}3$ & $0.157\mathrm{e}{-}3$
& $0.763\mathrm{e}{-}4$ & $0.871\mathrm{e}{-}4$ & $0.411\mathrm{e}{-}3$ & $0.273\mathrm{e}{-}4$\\
$|x_{2}-x^{\ast}|$ & $0.582\mathrm{e}{-}32$& $0.119\mathrm{e}{-}33$
& $0.540\mathrm{e}{-}35$ & $0.134\mathrm{e}{-}34$ & $0.377\mathrm{e}{-}29$ & $0.321\mathrm{e}{-}38$\\
$|x_{3}-x^{\ast}|$ & $0.532\mathrm{e}{-}261$ &
 $0.138\mathrm{e}{-}274$
& $0.342\mathrm{e}{-}284$ & $0.438\mathrm{e}{-}281$ & $0.189\mathrm{e}{-}237$& $0.117\mathrm{e}{-}309$ \\
COC &  $8.0000$ &  $8.0000$ & $8.0000$ & $8.0000$ & $8.0000$& $8.0000$ \\
ACOC & $8.0000$ & $7.9998$ & $7.9999$ & $7.9999$ & $7.9999$& $7.9999$ \\
\midrule
$f_3$, $x_0=-1.1$\\
$|x_{1}-x^{\ast}|$ & $0.106\mathrm{e}{-}7$ & $0.614\mathrm{e}{-}8$
& $0.388\mathrm{e}{-}8$ & $0.175\mathrm{e}{-}8$ & $0.554\mathrm{e}{-}8$ & $0.100\mathrm{e}{-}7$  \\
$|x_{2}-x^{\ast}|$ & $0.482\mathrm{e}{-}63$ & $0.328\mathrm{e}{-}65$
& $0.254\mathrm{e}{-}67$ & $0.154\mathrm{e}{-}70$ & $0.426\mathrm{e}{-}66$& $0.136\mathrm{e}{-}63$ \\
$|x_{3}-x^{\ast}|$ & $0.833\mathrm{e}{-}506$ &
 $0.217\mathrm{e}{-}523$
& $0.877\mathrm{e}{-}541$ & $0.582\mathrm{e}{-}567$ & $0.528\mathrm{e}{-}531$ & $0.154\mathrm{e}{-}510$\\
COC &  $8.0000$  & $8.0000$ & $8.0000$ & $8.0000$ & $8.0000$& $8.0000$ \\
ACOC & $7.9999$  & $8.0000$ & $7.9999$ & $8.0000$ & $8.0000$& $7.9999$ \\
\midrule
$f_4$, $x_0=1.5$\\
$|x_{1}-x^{\ast}|$ & $0.148\mathrm{e}{-}7$ & $0.433\mathrm{e}{-}8$
& $0.327\mathrm{e}{-}10$ & $0.642\mathrm{e}{-}10$ & $0.281\mathrm{e}{-}8$ & $0.727\mathrm{e}{-}10$ \\
$|x_{2}-x^{\ast}|$ & $0.138\mathrm{e}{-}61$& $0.134\mathrm{e}{-}66$
& $0.369\mathrm{e}{-}84$ & $0.101\mathrm{e}{-}81$ & $0.341\mathrm{e}{-}68$ & $0.543\mathrm{e}{-}81$ \\
$|x_{3}-x^{\ast}|$ & $0.769\mathrm{e}{-}494$ &
 $0.116\mathrm{e}{-}534$
& $0.967\mathrm{e}{-}676$ & $0.389\mathrm{e}{-}656$ & $0.161\mathrm{e}{-}547$& $0.530\mathrm{e}{-}650$ \\
COC &  $8.0000$  & $8.0000$ & $8.0000$ & $8.0000$ & $8.0000$ & $8.0000$ \\
ACOC & $8.0000$ & $7.9999$ & $7.9999$ & $7.9999$ & $8.0000$ & $7.9999$\\
\bottomrule
\end{tabular}
\caption{Errors, COC, and ACOC for the iterative methods
\eqref{a1} and \eqref{o1}--\eqref{o5} (abbreviated as M1--M6)
applied to the find the root of test functions $f_1,\dots,f_4$
given in Table~\ref{table1}.} \label{table2-3}
\end{center}
\end{table}


\section{Dynamic behavior}
\label{sec:dynamic}

We have already observed that all methods converge if the initial
guess is chosen suitably. We now investigate the regions where the
initial point has to chosen in order to achieve the root. In other
words, we will numerically approximate the domain of attraction of
the zeros as a qualitative measure of how the method depends on
the choice of the initial approximation of the root. To answer
this important question on the dynamical behavior of the
algorithms, we will investigate the dynamics of the new
method~\eqref{a1} and compare it with the methods
\eqref{o1}--\eqref{o5}.

Let's recall some basic concepts such as basin of attraction. For
more details and many other examples of the study of the dynamic
behavior of iterative methods, one can consult~\cite{Amat5,
Babajee, Chicharro, Ezquerro2, Ezquerro1, Fer, GutiMaVar, Paricio,
Stewart, Varona}.

Let $Q:\mathbb{C} \to \mathbb{C}$ be a rational map on the
complex plane. For $z\in \mathbb{C} $, we define its orbit as the
set $\operatorname{orb}(z) = \{z,\,Q(z),\,Q^2(z),\dots\}$. The convergence
$\operatorname{orb}(z)\to z^{\ast}$ is understood in the sense
$\lim\limits_{k\to\infty} Q^k(z) = z^{\ast}$. A point $z_0 \in
\mathbb{C}$ is called periodic point with minimal period $m$ if
$Q^m(z_0) = z_0$ where $m$ is the smallest positive integer with this
property (and thus $\{z_0,Q(z_0),\dots, Q^{m-1}(z_0)\}$ is a cycle).
A periodic point with minimal period $1$ is called fixed
point.
Moreover, a periodic point $z_0$ with period $m$ is called attracting if
$|(Q^m)'(z_0)|<1$, repelling if $|(Q^m)'(z_0)|>1$, and neutral
otherwise. The Julia set of a nonlinear map $Q(z)$, denoted by
$J(Q)$, is the closure of the set of its repelling periodic
points. The complement of $J(Q)$ is the Fatou set $F(Q)$.

The six methods \eqref{a1} and \eqref{o1}--\eqref{o5}
provide iterative rational maps $Q$ when they are applied to
find roots of complex polynomials $p$. In particular, we
are interested in the basins of attraction of the roots of the
polynomials where the basin of attraction of a root $z^{\ast}$ is the
complex set
$\{ z_0 \in \mathbb{C} : \operatorname{orb}(z_0)  \to z^{\ast} \}$.
It is well known that the basins of attraction of the
different roots lie in the Fatou set $F(Q)$. The Julia set $J(Q)$
is, in general, a fractal and the rational map $Q$ is unstable there.

For the dynamical and graphical point of view, we take a
$600 \times 600$ grid of the square $[-3,3]\times[-3,3] \subset \mathbb{C}$
and assign a color to each point $z_0\in D$ according to the simple
root to which the corresponding orbit of the iterative method
starting from $z_0$ converges, and we mark the point as black if
the orbit does not converge to a root in the sense that after at
most $15$ iterations it has a distance to any of the roots which
is larger than $10^{-3}$. We have used only $15$ iterations because
we are using methods of convergence order eight which, if they
converge, do this very fast. The basins of attraction are distinguished
by their color.

\begin{table}[htb!]
\centering
\begin{tabular}{l @{\qquad} c}
   \toprule
    Test polynomials & Roots
   \\
   \midrule
   $p_1(z) = z^2-1$ & $1,\quad -1$
  \\[0.5ex]
   $p_2(z) = z^3-z $ & $0,\quad 1,\quad -1$
  \\[0.5ex]
   $p_3(z) = z(z^2+1)(z^2+4) $ & $0, \quad 2i, \quad -2i,\quad i, \quad -i$
  \\[0.5ex]
   $p_4(z) = (z^4-1)(z^2+2i)$ & $1, \quad i, \quad -1,\quad -i, \quad -1+i, \quad 1-i$
  \\[0.5ex]
   $p_5(z) = z^7-1$ & $e^{2k\pi i/7}, \qquad k=0,\dots,6$
  \\[0.5ex]
   $p_6(z) = (10z^5-1)(z^5+10)$ & ${\big(\tfrac{1}{10}\big)}^{1/5} e^{2k\pi i/5},
         \quad (-10)^{1/5} e^{2k\pi i/5}, \qquad k=0,\dots,4$
  \\[0.5ex]
  \bottomrule
\end{tabular}
\caption{Test polynomials $p_1(z),\dots, p_6(z)$ and their roots.}
\label{table4}
\end{table}

Different colors are used for different roots. In the basins of attraction,
the number of iterations needed to achieve the root is shown by the
brightness. Brighter color means less iteration steps. Note that black
color denotes lack of convergence to any of the roots. This happens, in
particular, when the method converges to a fixed point that is not a root or
if it ends in a periodic cycle or at infinity. Actually and although we have
not done it in this paper, infinity can be considered an ordinary point if
we consider the Riemann sphere instead of the complex plane. In this case,
we can assign a new ``ordinary color'' for the basin of attraction of
infinity. Details for this idea can be found in~\cite{Paricio}.

\begin{figure}
\centering
\includegraphics[width=0.9\textwidth]{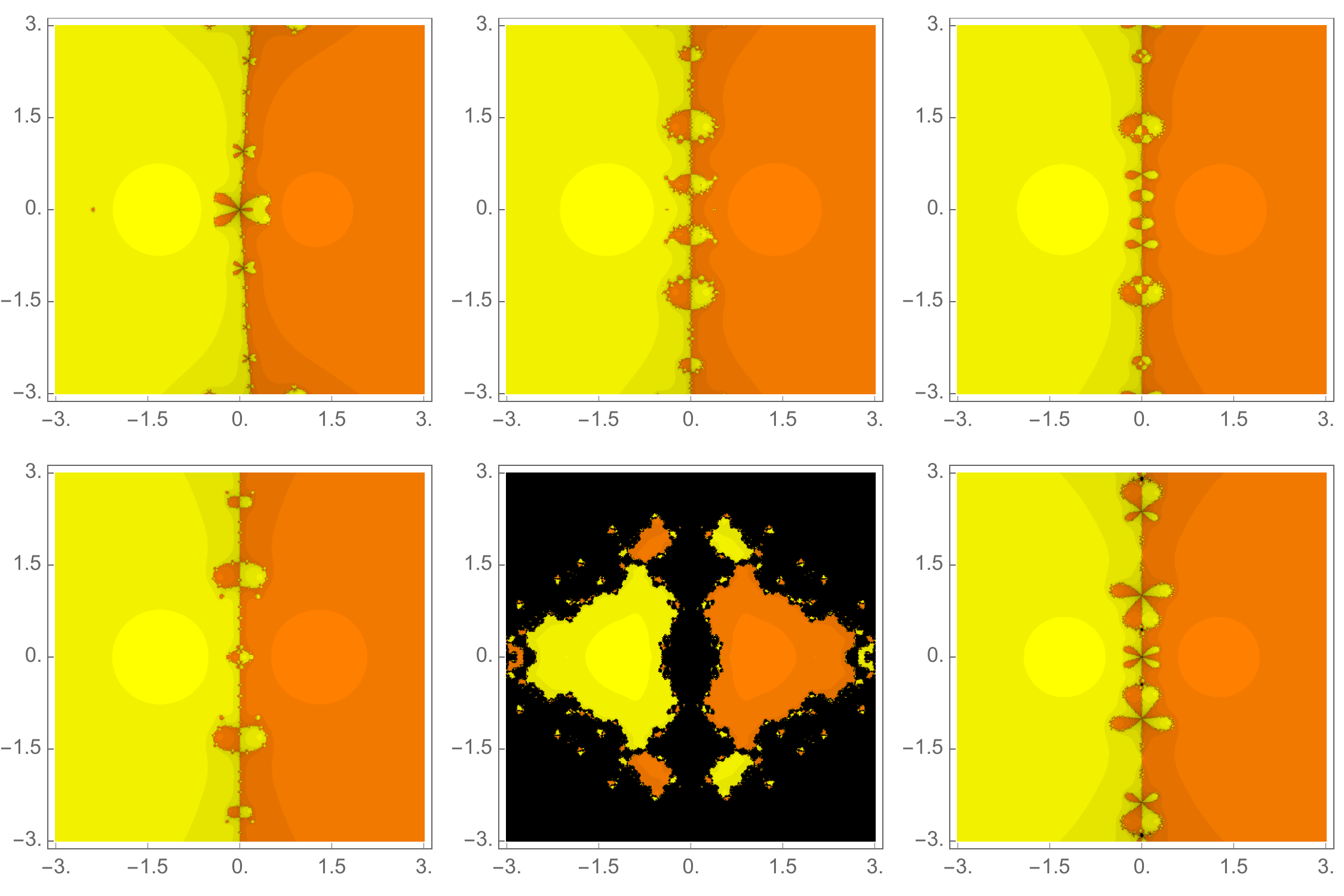}
\caption{Comparison of basins of attraction of methods \eqref{a1}
and \eqref{o1}--\eqref{o5} for the test problem $p_1(z)=
z^2-1=0$.} \label{fig:figure1}
\end{figure}

\begin{figure}
\centering
\includegraphics[width=0.9\textwidth]{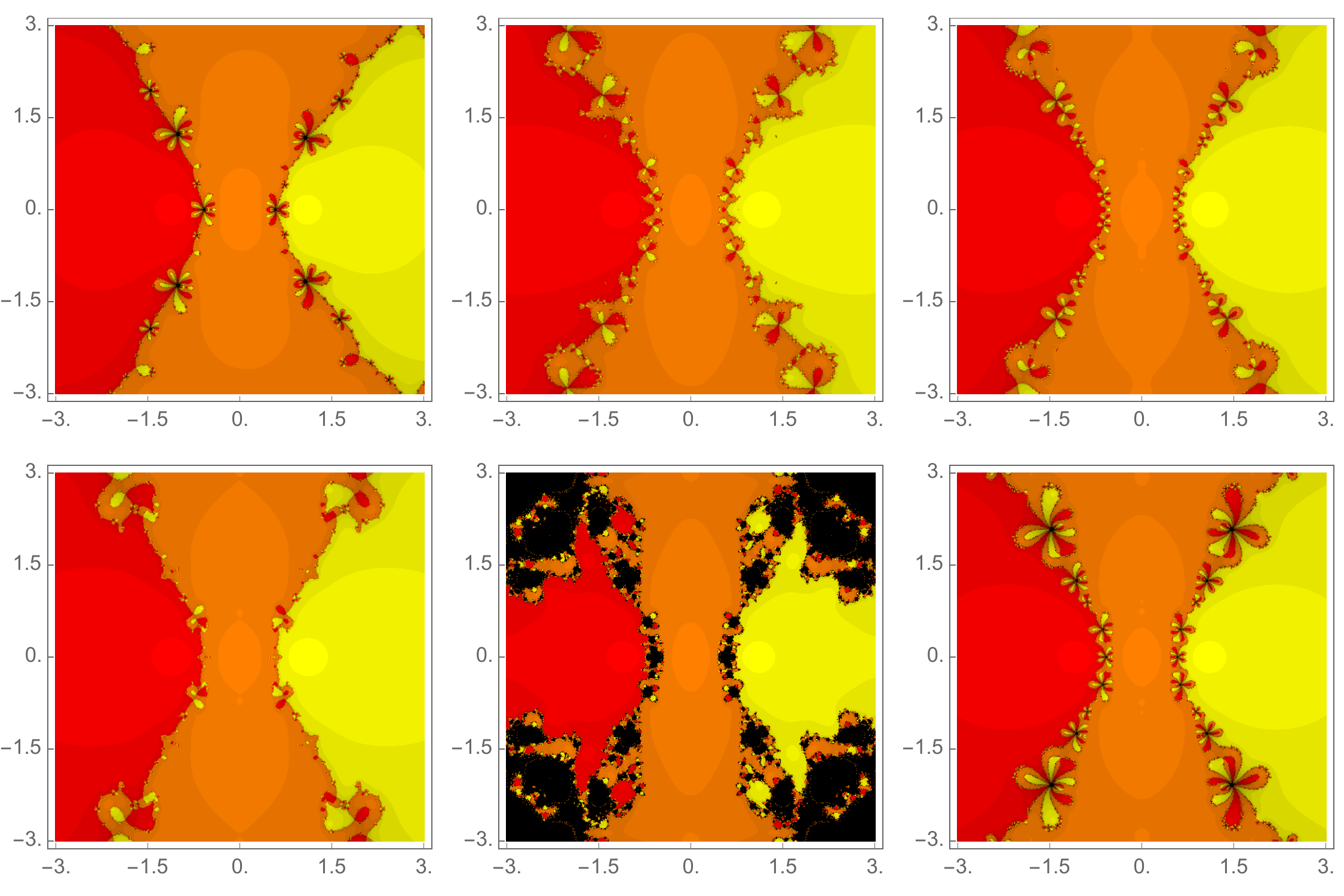}
\caption{Comparison of basins of attraction of methods \eqref{a1}
and \eqref{o1}--\eqref{o5} for the test problem $p_2(z)=
z^3-z=0$.} \label{fig:figure2}
\end{figure}

Basins of attraction for the six methods \eqref{a1} and
\eqref{o1}--\eqref{o5} for the six test problems $p_i(z)=0$,
$i=1,\dots,6$, are illustrated in
Figures~\ref{fig:figure1}--\ref{fig:figure6} from left to right
and from top to bottom.

\begin{figure}
\centering
\includegraphics[width=0.9\textwidth]{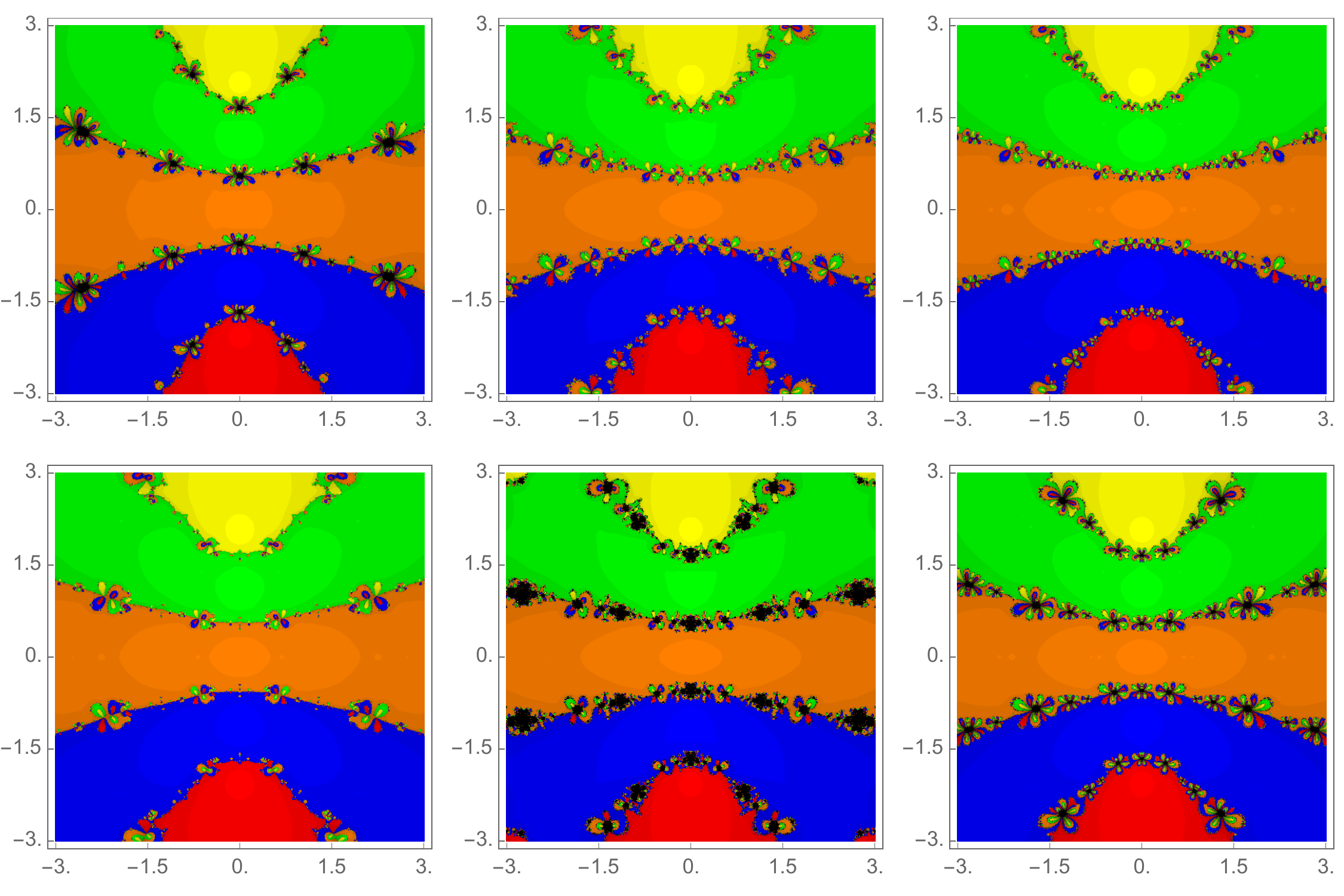}
\caption{Comparison of basins of attraction of methods \eqref{a1}
and \eqref{o1}--\eqref{o5} for the test problem $p_3(z)=
z(z^2+1)(z^2+4)=0$.} \label{fig:figure3}
\end{figure}

\begin{figure}
\centering
\includegraphics[width=0.9\textwidth]{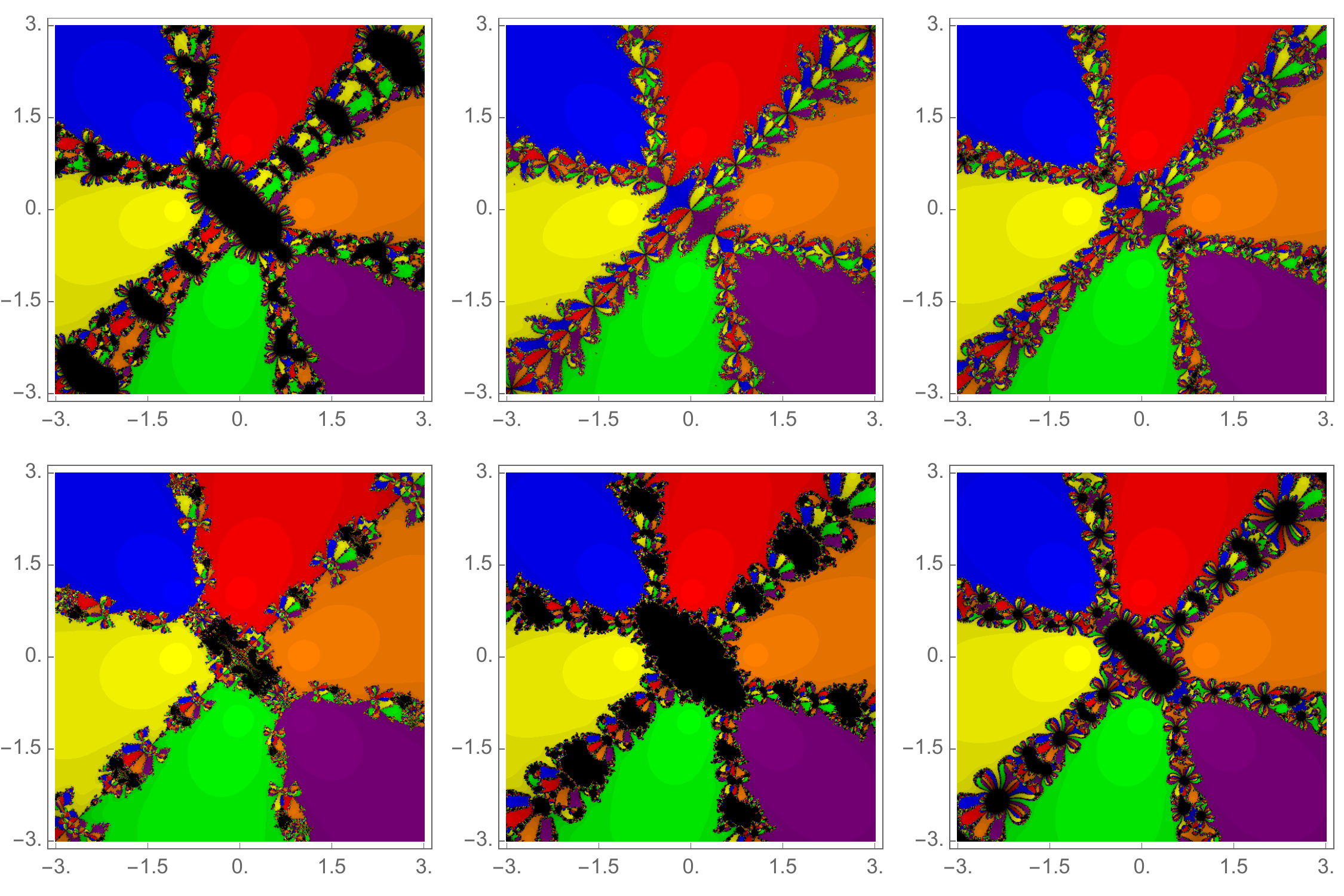}
\caption{Comparison of basins of attraction of methods \eqref{a1}
and \eqref{o1}--\eqref{o5} for the test problem $p_4(z)=
(z^4-1)(z^2+2i)=0$.} \label{fig:figure4}
\end{figure}

From the pictures, we can easily judge the behavior and suitability
of any method depending on the circumstances.
If we choose an initial point $z_0$
in a zone where different basins of attraction touch each other, it is
impossible to predict which root is going to be reached by the
iterative method that starts in~$z_0$. Hence, $z_0$ is
not a good choice. Both the black zones and the zones with a lot
of colors are not suitable for choosing the initial guess $z_0$ if
precise root should be reached. Although the most attractive pictures appear
when we have very intricate frontiers between basins of attraction, they
correspond to the cases where the dynamic behavior of the method is more
unpredictable and the method is more demanding with respect to choice of
the initial point.

\begin{figure}
\centering
\includegraphics[width=0.9\textwidth]{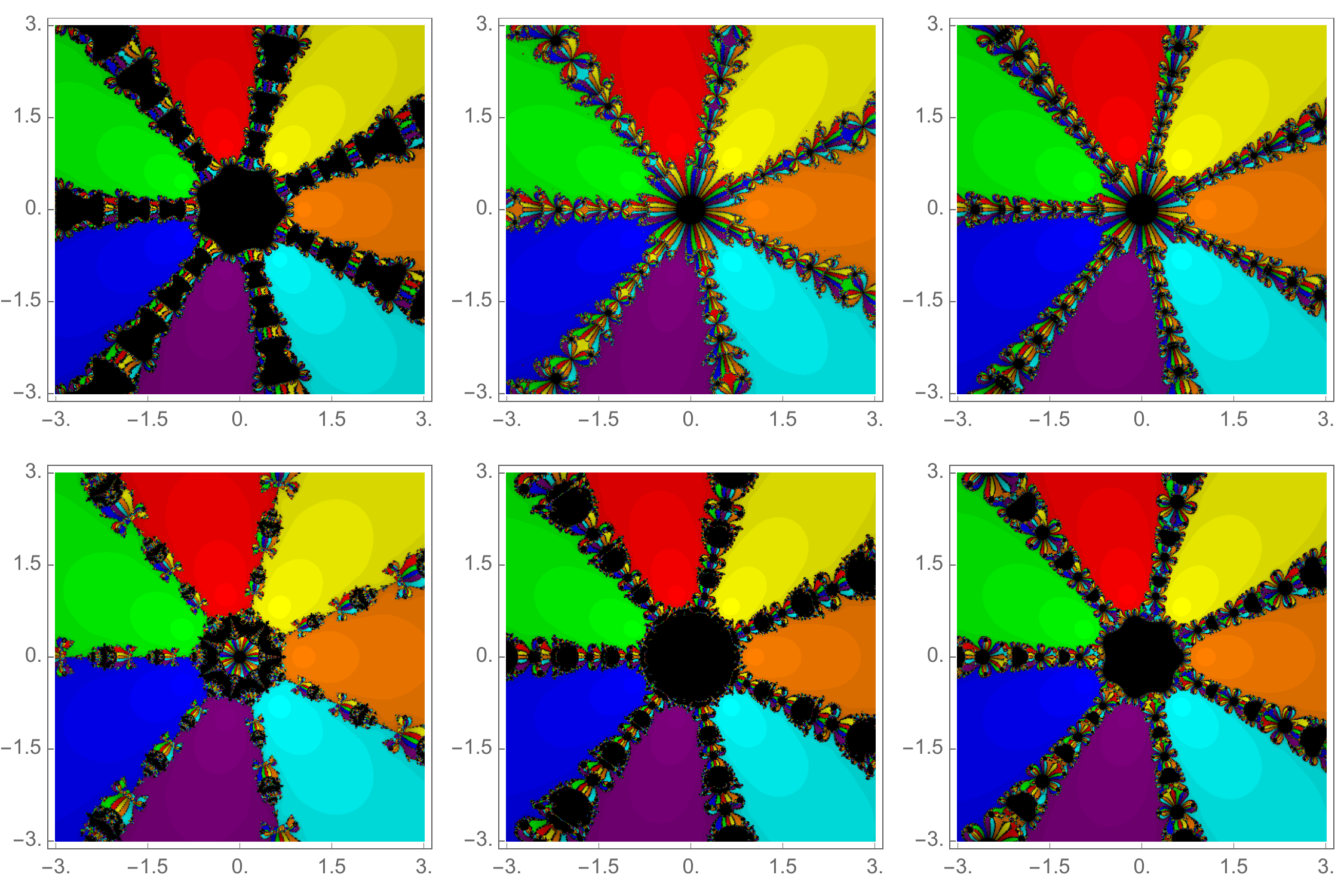}
\caption{Comparison of basins of attraction of methods \eqref{a1}
and \eqref{o1}--\eqref{o5} for the test problem $p_5(z)=
z^7-1=0$.} \label{fig:figure5}
\end{figure}

\begin{figure}
\centering
\includegraphics[width=0.9\textwidth]{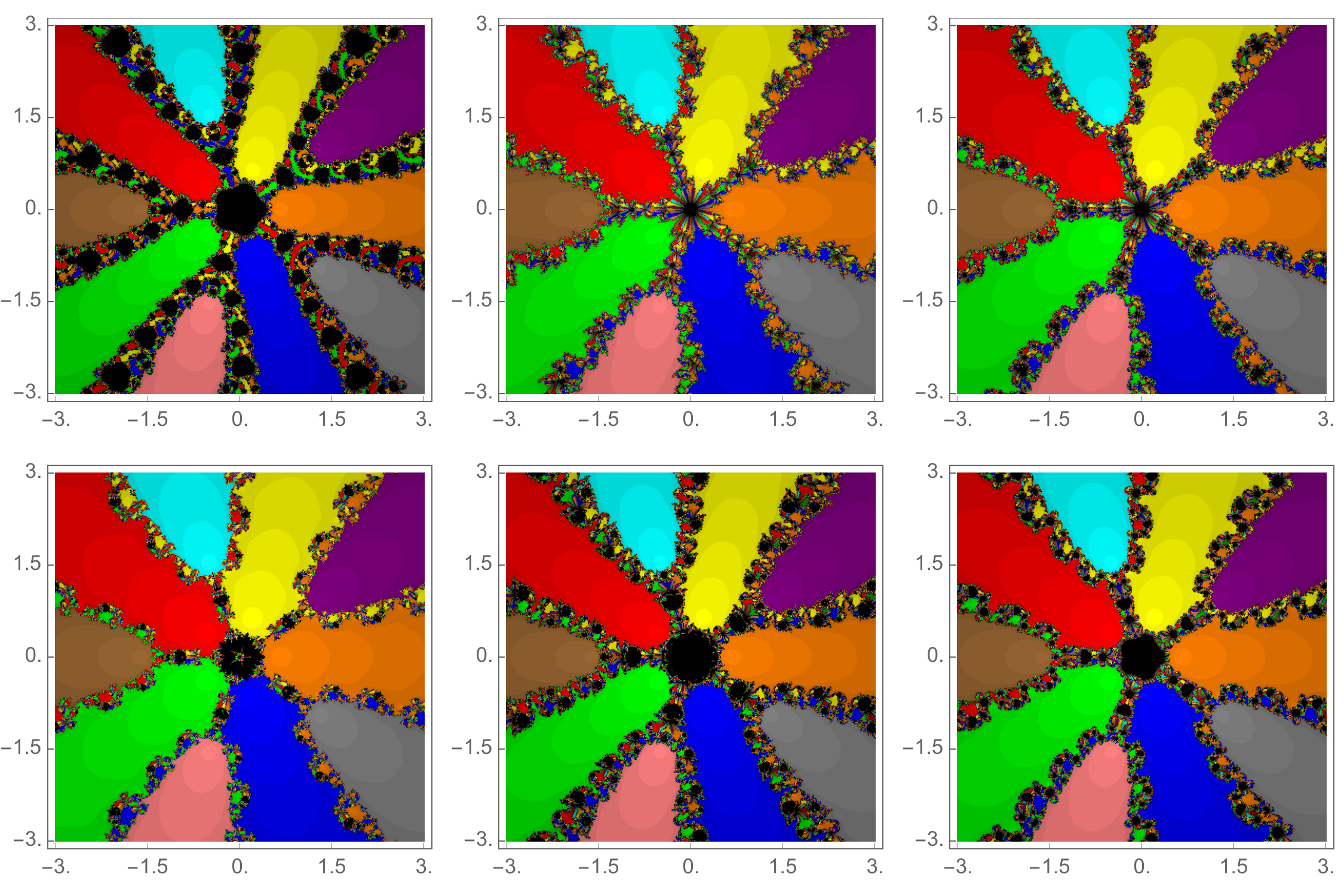}
\caption{Comparison of basins of attraction of methods \eqref{a1}
and \eqref{o1}--\eqref{o5} for the test problem $p_6(z)=
(10z^5-1)(z^5+10)=0$.} \label{fig:figure6}
\end{figure}

Finally, we have included in Table~\ref{table5} the results of
some numerical experiments to measure the behavior of the five
iterative methods \eqref{a1} and \eqref{o1}--\eqref{o5} in finding
the roots of the test polynomials $p_j$, $j=1,\dots,6$. To
compute the data of this table, we have applied the six methods to
the six polynomials, starting at an initial points $z_0$ on a $600
\times 600$ grid in the rectangle $[-3,3] \times [-3,3]$ of the
complex plane. The same way was used in
Figures~\ref{fig:figure1}--\ref{fig:figure6} to show the basins of
attraction of the roots. In particular, we decide again that an
initial point $z_0$ has reached a root $z^{\ast}$ when its distance to
$z^{\ast}$ is less than $10^{-3}$ (in this case $z_0$ is in the basin
of attraction of~$z^{\ast}$) and we decide that the method starting in
$z_0$ diverges when no root is found in a maximum of $15$
iterations of the method. We say in this case that $z_0$ is a
``nonconvergent point''. In Table~\ref{table5}, we have
abbreviated the methods \eqref{a1} and \eqref{o1}--\eqref{o5} as
M1--M6, respectively. The column I/P shows the mean of iterations
per point until the algorithm decides that a root has been reached
or the point is declared nonconvergent. The column NC shows the
percentage of nonconvergent points, indicated as black zones in
the pictures of Figures~\ref{fig:figure1}--\ref{fig:figure6}.
It is clear that the nonconvergent points have a
great influence on the values of I/P since these points contribute
always with the maximum number of $15$ allowed iterations. In
contrast, ``convergent points'' are reached usually very fast due
to the fact that we are dealing with methods of order eight. To
reduce the effect of nonconvergent points, we have included the
column I$_\textrm{C}$/C which shows the mean number of iterations
per convergent point. If we use either the columns I/P or the
column I$_\textrm{C}$/C to compare the performance of the
iterative methods, we clearly obtain different conclusions.


\begin{table}[htbp!]
\begin{center}
\begin{tabular}{cclll}
\toprule
Polynomial & Method & I/P & NC (\%) & I$_\textrm{C}$/C \\
\midrule
$p_1(z)$ & M1 & 2.21 & 0.00111 & 2.21 \\
         & M2 & 2.19 & 0. & 2.19 \\
         & M3 & 2.16 & 0. & 2.16 \\
         & M4 & 2.11 & 0. & 2.11 \\
         & M5 & 6.01 & 71.0 & 2.09 \\
         & M6 & 2.30 & 0.0256 & 2.30 \\
\midrule
$p_2(z)$ & M1 & 2.90 & 0.125 & 2.89 \\
         & M2 & 2.88 & 0.00111 & 2.88 \\
         & M3 & 2.82 & 0.00444 & 2.82 \\
         & M4 & 2.73 & 0. & 2.73 \\
         & M5 & 4.32 & 27.5 & 2.81 \\
         & M6 & 3.21 & 0.216 & 3.18 \\
\midrule
$p_3(z)$ & M1 & 3.22 & 0.802 & 3.13 \\
         & M2 & 2.99 & 0.0178 & 2.99 \\
         & M3 & 2.94 & 0.0367 & 2.94 \\
         & M4 & 2.82 & 0. & 2.82 \\
         & M5 & 3.28 & 5.47 & 2.99 \\
         & M6 & 3.42 & 1.08 & 3.30 \\
\bottomrule
\end{tabular}%
\quad
\begin{tabular}{cclll}
\toprule
Polynomial & Method & I/P & NC (\%) & I$_\textrm{C}$/C \\
\midrule
$p_4(z)$ & M1 & 6.00 & 17.7 & 4.06 \\
         & M2 & 4.06 & 0.819 & 3.97 \\
         & M3 & 4.21 & 1.82 & 4.01 \\
         & M4 & 3.95 & 4.40 & 3.44 \\
         & M5 & 4.44 & 20.0 & 3.57 \\
         & M6 & 5.17 & 9.35 & 4.15 \\
\midrule
$p_5(z)$ & M1 & 6.89 & 24.4 & 4.27 \\
         & M2 & 4.81 & 3.33 & 4.46 \\
         & M3 & 5.07 & 5.70 & 4.46 \\
         & M4 & 4.59 & 7.05 & 3.80 \\
         & M5 & 5.02 & 21.4 & 4.02 \\
         & M6 & 5.78 & 13.3 & 4.36 \\
\midrule
$p_6(z)$ & M1 & 6.72 & 18.2 & 4.88 \\
         & M2 & 4.68 & 2.29 & 4.44 \\
         & M3 & 4.89 & 4.04 & 4.46 \\
         & M4 & 4.44 & 3.96 & 4.00 \\
         & M5 & 5.26 & 11.8 & 4.71 \\
         & M6 & 5.45 & 8.49 & 4.56 \\
\bottomrule
\end{tabular}
\caption{Measures of convergence of the iterative methods
\eqref{a1} and \eqref{o1}--\eqref{o5} (abbreviated as M1--M6)
applied to find the roots of the polynomials $p_j(z)$,
$j=1,\dots,6$.} \label{table5}
\end{center}
\end{table}

\section{Conclusion}
\label{sec:conclusion}

We have introduced a new optimal three-point method without memory
for approximating a simple root of a given nonlinear equation
which use only four function evaluations each iteration and result
in a method of convergence order eight. Therefore, Kung and
Traub's conjecture is supported. Numerical examples and
comparisons with some existing eighth-order methods are included
and confirm the theoretical results. The numerical experience
suggests that the new method is a valuable alternative for solving
these problems and finding simple roots. We used the basins of
attraction for comparing the iterative algorithms and we have
included some tables with comparative results.

\paragraph*{Acknowledgments.}
The research of the fourth author is supported by grant
MTM2015-65888-C4-4 from DGI (Spanish Government).


\end{document}